\documentclass[12pt]{extarticle}
\usepackage{extsizes}
\usepackage[reqno]{amsmath} % equation numbers on the right
\usepackage{amsmath}
\usepackage{amssymb}
\usepackage{amsthm}
\usepackage{amscd}
\usepackage{amsfonts}
\usepackage{graphicx}
\usepackage{dsfont}
\usepackage{fancyhdr}
\usepackage{enumerate}
\usepackage{youngtab}
\usepackage[utf8]{inputenc}
\usepackage{comment}

\usepackage{subfigure}
\usepackage[margin=1in]{geometry}
\usepackage{graphicx}
\usepackage[noend]{algorithmic}
\usepackage{algorithm}
\usepackage{color}
\usepackage{hyperref}
\usepackage{notation}

\theoremstyle{definition}
\newtheorem{corollary}{Corollary}

\newtheorem{lemma}[corollary]{Lemma}

\newtheorem{theorem}[corollary]{Theorem}

\begin{document}

\title{Algebras, graphs and thetas}

\author{\small{\bf Marcel K. de Carli Silva}\footnote{M. de Carli Silva acknowledges support from CNPq (Brazil), Proc.~477203/2012-4 and 456792/2014-7.} \\ \small{Dept.\ of Computer Science} \\ \small{University of São Paulo} \\ \small{São Paulo, Brazil} \\ \small{mksilva@ime.usp.br} \and
\small{\bf Gabriel Coutinho} \\ \small{Dept.\ of Computer Science} \\ \small{Federal University of Minas Gerais} \\ \small{Belo Horizonte, Brazil}\\ \small{gabriel@dcc.ufmg.br} \and
\small{\bf Chris Godsil}\footnote{C. Godsil acknowledges support of NSERC (Canada), Grant No.\ RGPIN-9439.} \\ \small{Dept.\ of Combinatorics}\\\small{and Optimization}\\\small{University of Waterloo} \\ \small{Waterloo, Canada} \\ \small{cgodsil@uwaterloo.ca}\and
\small{\bf David E. Roberson} \\ \small{Dept.\ of Applied Mathematics}\\\small{and Computer Science} \\ \small{Technical University of Denmark} \\ \small{Kongens Lyngby, Denmark} \\ \small{davideroberson@gmail.com}}

\date{\small{\today}}
\maketitle
\vspace{-0.8cm}
\begin{abstract} 
    We extend the clique-coclique inequality, previously known to hold for graphs in association schemes and vertex-transitive graphs, to graphs in homogeneous coherent configurations and 1-walk regular graphs.  We further generalize it to a stronger inequality involving the Lovász theta number of such graph, and some theta variants, including characterizations of the equality.
    \begin{center}
    \textbf{Keywords}
    \end{center}
    clique-coclique bound, matrix $*$-algebra, Lovász theta function, coherent configuration
    
\end{abstract}

\section{Introduction}

Denote the maximum size of a clique in a graph~$G$ by $\omega(G)$ and the maximum size of a coclique by~$\alpha(G)$; a coclique is an independent set of vertices, also called a stable set.

Think of a graph with a large clique and a large coclique. In general, there seems to be no (non-obvious) restriction on how large these two substructures can be compared to the size of the graph. However, if the graph displays high regularity or symmetry, then one soon finds out that large cliques cannot appear together with large cocliques, and vice-versa. More specifically, for a graph $G$ on $n$ vertices that is either  distance-regular or vertex-transitive, we have $\alpha(G) \omega(G) \leq n$. This is known as the clique-coclique bound; see for instance \cite[Chapter 3]{GodsilMeagherEKRBook}. In this paper, we observe that a more general framework in which these graphs and their cliques and cocliques can be cast is sufficient to prove the clique-coclique bound. This appears in Theorems \ref{thm:coherent} and \ref{thm:walkreg}, and Corollary \ref{cor:cliquecoclique}.

Our setting also allows for an extension of this inequality to a stronger inequality (which turns out to be an equality) involving the Lovász theta number of a graph and its complement, and another equality for some theta variants. This is found in Corollary \ref{cor:eqthetas}. These results were known for vertex-transitive graphs only, therefore our contribution significantly increases the number of known graphs satisfying the properties displayed.

Some work has been done in the past in topics strongly related to the topic of this paper. Godsil and Meagher \cite{GodsilMeagherEKRBook} present a full account of known bounds for the size of cliques and cocliques in graphs belonging to association schemes. Dukanovic and Rendl \cite{DukanovicRendl} proved some equalities involving generalizations of thetas for a vertex-transitive graph and its complement. Roberson \cite{RobersonConicFormulations} pointed out that we could somehow relax the requirement of the graph being vertex-transitive. Some of the results in our paper work in the direction of finding exactly which graphs satisfy this relaxed condition. The use of positive semidefiniteness to find bounds in combinatorial structures dates back to Delsarte's thesis \cite{DelsarteAssocSchemesCoding}. It was revived more recently by Schrijver \cite{SchrijverNewCodeUpperBounds} in the context of coding theory and later applied to some coherent configurations by Hobart \cite{HobartBoundsSubsetsCoherent} and Hobart and Williford \cite{HobartWillifordTightnessBoundsSubsetsCoherent}. Some surveys can be found in \cite{AnjosLasserreHandbook}. In all cases, the main tool is the fact that the projection of a positive semidefinite matrix in certain algebras of matrices remains positive semidefinite. We explore this fact in Lemmas \ref{lem1} and \ref{lem2} to build our theoretical framework.

\section{Algebras}

Equip the complex vector space $\textrm{M}_n(\C)$ of complex $n \times n$ matrices with the trace inner-product $\langle M, N\rangle = \tr MN^*$.  We will work on matrix $*$-algebras, which means a linear subspace of $\C^{n \times n}$ that is closed under the conventional matrix product and under taking the conjugate transpose.  We start with the following consequence of Wedderburn's Theorem on semisimple algebras (\cite[Chapter 5]{CohnBasicALgebra}); see for instance \cite[Theorem~2.7]{Bachoc} for a more self-contained version, or \cite[Chapter 1]{ArvesonInvitation}.

\begin{theorem}\label{thm:decomp}
  If $\A$ is a matrix $*$-algebra, then $\A$ is the direct sum of simple matrix $*$-algebras
  \begin{align}\label{eq:1}
    \A = \bigoplus_{i = 0}^s \A_i.
  \end{align}
\end{theorem}

Given a matrix $M$, we will use $M'$ to denote the orthogonal projection of $M$ onto $\A$. Because of the decomposition from above, it follows that $M'$ is the sum of the projection of $M$ onto each $\A_i$. The following fact is also well-known, see for instance \cite[Corollary 9.1]{AnjosLasserreHandbook}. 

\begin{corollary}\label{cor:psd}
  The projection of a positive semidefinite matrix onto a matrix $*$-algebra is positive semidefinite.
\end{corollary}

From here on, we assume all matrices in $\A$ have constant diagonal. This property shall be referred as $\A$ being \textit{homogeneous}.

\begin{lemma}\label{lem:0}
  If $\A$ is a homogeneous matrix $*$-algebra, then all $01$ matrices in $\A$ have constant row sums and constant column sums. 
  
  Moreover, if $\A$ contains an irreducible $01$ matrix, then the all $1$s matrix $J$ belongs to $\A$.
  
  Finally, if $\A$ contains the all 1s matrix $J$, then the row sums and the column sums of all matrices in $\A$ are equal, that is, $J$ lies in the center of $\A$.
\end{lemma}
\begin{proof}
  Let $A \in \A$ be a $01$ matrix. Then $A^* \in \A$, and so $AA^* \in \A$ and $A^* A \in \A$. The diagonal entries of $AA^*$ are the row sums of $A$, and the diagonal entries of $A^* A$ are the column sums of $A$. Because $\A$ is homogeneous, all row sums are equal, and the same holds for column sums. Now if $A$ is irreducible, and because its row sums are constant, then by Perron-Frobenius theory the all $1$s vector is an eigenvector in a $1$-dimensional subspace, and therefore $J$ is a polynomial in $A$, hence it belongs to $\A$. The result now follows from noting that, for all $M \in \A$, the diagonal entries of $MJ$ and $JM$ are, respectively, the row sums and column sums of $M$. These are constant diagonals, and because $\tr MJ = \tr JM$, it follows that these diagonal are, in fact, equal.
\end{proof}

As $M \mapsto M'$ is a self-adjoint operator, it follows that
\[\langle M , N' \rangle = \langle M' , N\rangle,\]
a fact that we exploit below for two special cases.

\begin{lemma}
  Let $M$ be a matrix, and let $\A$ be a homogeneous matrix $*$-algebra
  that contains $I$ and $J$. Then
  \[\tr M' = \tr M,\]
  and
  \[ \tr J M' = \tr J M.\]
\end{lemma}
\begin{proof}
  Since $I$ and $J$ belong to $\A$, we have $I' = I$ and $J' = J$.  Thus, $\tr M' = \langle M' , I \rangle = \langle M , I' \rangle = \tr M$ and $\tr J M' = \tr M' J = \langle M' , J \rangle  = \langle M , J' \rangle = \tr M J = \tr J M$.
\end{proof}

We show below how to obtain a trace inequality in a special but useful case.

\begin{lemma} \label{lem1}
  Let $\A$ be a homogeneous matrix $*$-algebra that contains~$I$. Let $M$ and $N$ be positive semidefinite $n \times n$ matrices. Let $I = P_0 , P_1 , \dotsc , P_d$ be an orthogonal basis for the algebra $\A$. Assume that, for all $i \neq 0$, $\langle M , P_i \rangle \langle P_i, N \rangle \leq 0$. Then
  \[\langle M',N' \rangle \leq \frac{(\tr M)(\tr N)}{n}.\]
  Moreover, equality holds if and only if, for all $i \neq 0$, $\langle M , P_i \rangle \langle P_i, N \rangle = 0$.
\end{lemma}
\begin{proof}
  It is a straightforward computation:
  \begin{align*}
    \langle M', N'\rangle = \langle M' , N \rangle & = \tr \left(\sum_{i = 0}^d \frac{\langle M , P_i \rangle}{\langle P_i , P_i \rangle} P_i \right)N^*  \\
                                                   & = \sum_{i = 0}^d \frac{\langle M , P_i \rangle \langle P_i, N \rangle}{\langle P_i , P_i \rangle} \\
                                                   & \leq \frac{\langle M , I \rangle \langle I, N \rangle}{\langle I , I\rangle }\\
                                                   & = \frac{(\tr M)(\tr N)}{n}.
  \end{align*}
  The equality characterization follows immediately.
\end{proof}

\begin{lemma}\label{lem2}
  Let $\A$ be a homogeneous matrix $*$-algebra that contains $I$ and $J$. Let $M$ and $N$ be positive semidefinite $n \times n$ matrices. Then
  \[\langle M',N'\rangle \geq \frac{(\tr J M) (\tr J N)}{n^2}.\]
  Moreover, equality holds if and only if $M' N'$ is a multiple of $J$.
\end{lemma}
\newcommand{\lambd}{i}
\begin{proof}
  By Lemma \ref{lem:0} the all $1$s vector is an eigenvector of all matrices in $\A$, and $J$ lies in the center of $\A$. So the the multiples of $J$ form an $\A$-module, and, in particular, one of the factors in Theorem \ref{thm:decomp} is equal to this module. We may assume this is $\A_0$. Let $M_i$ denote the projection of $M$ onto the $\A_i$s from Theorem \ref{thm:decomp}, and similarly for $N$ and $N_i$.
  \[M' = \sum_{\lambd = 0}^s M_\lambd \quad \text{and}\quad N' = \sum_{\lambd = 0}^s N_\lambd,\]
  and the matrices $M_i$ and $N_i$ are positive semidefinite by Corollary \ref{cor:psd}. Moreover, the orthogonality of the decomposition implies that
  \[M' N' = \sum_{\lambd = 0}^s M_\lambd N_\lambd.\]
  Recall that the trace of the product of positive semidefinite matrices is non-negative. Therefore
  \[\tr M' N' = \sum_{\lambd = 0}^s \tr M_\lambd N_\lambd \geq \tr M_0 N_0 = \frac{(\tr J M) (\tr J N)}{n^2}.\]
  Equality holds if and only if $M_\lambd N_\lambd = 0$ for all $\lambd \neq 0$, which is equivalent to $M' N'$ being a multiple of $J$.
\end{proof}

Denote the Schur (componentwise) product of matrices $B$ and $C$ by $B \circ C$.
In what follows, we will typically consider positive semidefinite matrices $M$ and $N$, that with respect to the adjacency matrix $A$ of a graph~$G$ and the adjacency matrix $\overline{A} = J - I - A$ of the complement~$\overline{G}$, satisfy either of the following two conditions:
\begin{gather}
M \circ A = 0 \text{ and }N \circ \overline{A} = 0, \tag{A}\label{i}
\\
M \circ A \leq 0,\; N \circ \overline{A} = 0\text{ and }N \circ A \geq 0. \tag{B}\label{ii}
\end{gather}

	\section{Graphs}
	
	We now show two classes of examples of matrix $*$-algebras satisfying the properties of the lemmas above.
	
	\subsection{Homogeneous coherent configurations}
	
	A \textit{coherent configuration} is a finite set of $01$ matrices $\{A_0,...,A_d\}$, which satisfies the following properties:
	\begin{enumerate}[(i)]
	    \item $\sum_{i = 0}^d A_i = J$.
	    \item For all $i \in \{0,...,d\}$, if one diagonal entry of $A_i$ is non-zero, then $A_i$ is diagonal.
	    \item The configuration is transpose-closed.
	    \item $A_i A_j$ is a linear combination of the matrices in the configuration, for all $i$ and $j$.
	\end{enumerate}
	Coherent configurations appear naturally in connection to design theory, finite geometry, coding theory and representation of finite groups. They were originally defined by Higman in \cite{HigmanCoherentConfig}. When the identity matrix is one of the matrices forming the configuration, we call such configuration homogeneous. When the matrices forming the configuration commute, they result in what is known as an association scheme. The theory of association schemes is vast and rich, and the connections to combinatorics are overwhelming; see for instance \cite{BCN}.
	
	From here on, we assume $\{I = A_0,\dotsc,A_d\}$ is a homogeneous coherent configuration. These matrices generate a complex algebra, called the \textit{coherent algebra}, which we denote by $\A$. Note that this algebra is a homogeneous matrix $*$-algebra that contains $I$ and $J$. 
	
	A graph belongs to a coherent configuration if its adjacency matrix is a sum of the matrices forming the configuration. The so-called distance-regular graphs are standard examples of graphs found in (and generating) association schemes. Another class of examples comes from vertex-transitive graphs --- the permutation matrices corresponding to the automorphisms of the graph form a group, and its commutant in $\textrm{M}_n(\C)$, which contains the adjacency matrix of the graph, is a homogeneous coherent algebra. 
	
	If $A$ is the adjacency matrix of a graph and belongs to a homogeneous coherent algebra generated by the configuration $\{I = A_0,\dotsc,A_d\}$, it follows that, for some $R \subseteq \{1,\dotsc,d\}$, we have
	\[A = \sum_{r \in R} A_r.\]
	As a consequence,
	\[\overline{A} = \sum_{r \in \overline{R}} A_r,\]
	where $\overline{R} = \{1,\dotsc,d\} \setminus R$.
	Moreover, if $A_r^* = A_{r^*}$, then $r \in R$ implies $r^* \in R$.
	
	Therefore if $M$ and $N$ are positive semidefinite matrices satisfying either of the conditions \eqref{i} and \eqref{ii}, then it follows that $\langle M , A_r\rangle \langle N , A_r \rangle \leq 0$ for all $r \in \{1,\dotsc,d\}$, and thus $M$, $N$, $\A$ and its basis $\{A_0,\dotsc,A_d\}$ satisfy the conditions of Lemma \ref{lem1}.
	
	Applying Lemmas \ref{lem:0}, \ref{lem1} and \ref{lem2}, we have the result below.
	
	\begin{theorem}\label{thm:coherent}
		Let $A$ be the adjacency matrix of a connected graph that belongs to a homogeneous coherent configuration $\{I = A_0,\dotsc,A_d\}$. Let $M$ and $N$ be non-zero positive semidefinite matrices satisfying conditions (\ref{i}) or (\ref{ii}). Then
		\[n \geq \frac{(\tr JM) (\tr JN)}{(\tr M)(\tr N)}.\]
		Moreover, equality holds if and only if $M' N'$ is a scalar multiple of $J$, and either \eqref{i} holds or \eqref{ii} holds in such way that $\langle M, A_r \rangle \langle A_r, N \rangle = 0$ for all $r \neq 0$.
	\end{theorem}
	
	\subsection{1-walk regular graphs (and their complements)}
	
	A graph $G$ with adjacency matrix $A$ is called $1$-walk regular if, for any positive integer~$k$, $A^k$ is constant in the diagonal and in the entries corresponding to the support of $A$. That is, for any positive integer $k$, there are constants $a_k$ and $b_k$ such that
	\[A^k \circ I = a_k I \quad \text{and} \quad A^k \circ A = b_k A.\]
	
	Assume $G$ is $1$-walk regular, and $A$ is its adjacency matrix. Note for instance that $G$ must be regular, as the diagonal of $A^2$ is constant. Let $\A$ be the algebra generated by $A$ --- that is, the adjacency algebra of the graph $G$. This is obviously a matrix $*$-algebra. In fact, it is an algebra of symmetric and commuting matrices, therefore the decomposition of $\R^n$ into $\A$-modules is simply its simultaneous diagonalization.
	
	Consider an orthogonal basis $I = A_0 , A = A_1, A_2 ,\dotsc, A_d$. For instance, this could have been obtained by applying Gram-Schmidt to $I,A,A^2,\dotsc,A^d$.
	The relevant consequence of $1$-walk regularity is that $A_k \circ I = 0$ and $A_k \circ A = 0$ for all $k \geq 2$, as the basis is orthogonal and all matrices of $\A$ are constant over the support of $I$ and $A$. 
	
	Now let $M$ and $N$ be positive semidefinite matrices, and assume that \eqref{i} or \eqref{ii} holds. Then $\langle M , A_k \rangle \langle A_k , N \rangle \leq 0$ for all $k \geq 1$. Thus $M$, $N$, $\A$ and its basis $\{A_0,\dotsc,A_d\}$ satisfy the conditions of Lemmas \ref{lem1} and~\ref{lem2}, therefore we have the following theorem.
	
	\begin{theorem}\label{thm:walkreg}
		Let $A$ be the adjacency matrix of a $1$-walk regular graph. Let $M$ and $N$ be non-zero positive semidefinite matrices satisfying conditions (\ref{i}) or (\ref{ii}) above. Then
		\[n \geq \frac{(\tr JM) (\tr JN)}{(\tr M)(\tr N)}.\]
		Moreover, equality holds if and only if $M' N'$ is a scalar multiple of $J$, and either (\ref{i}) holds or (\ref{ii}) holds in such way that $\langle M, A_r \rangle \langle A_r , N \rangle = 0$ for all $r \neq 0$, where $I = A_0 , A = A_1, A_2 ,\dotsc, A_d$ is an orthogonal basis of the adjacency algebra of~$G$. 
	\end{theorem}
	
	\section{Thetas}
	
	The immediate corollaries to Theorems \ref{thm:coherent} and \ref{thm:walkreg} are the so-called clique-coclique bounds. Assume $S$ is a clique and $T$ is a coclique in a graph which is either in a homogeneous coherent configuration or is 1-walk regular. Let $\chi_S$ and $\chi_T$ be their respective characteristic vectors, and define
	\[N = \chi_S^{} \chi_S^* \quad \text{and} \quad M = \chi_T^{} \chi_T^*.\]
	It follows that $M \circ A = 0$ and $N \circ \ov A = 0$. Thus Theorem \ref{thm:coherent} or Theorem \ref{thm:walkreg} applies, and we obtain the following result, which has been stated before for graphs in association schemes or vertex-transitive graphs (see for instance \cite[Theorem 3.8.4]{GodsilMeagherEKRBook}).
	\begin{corollary} \label{cor:cliquecoclique}
		Assume $S$ is a clique and $T$ is a coclique in a graph on $n$ vertices which is either in a homogeneous coherent configuration or is 1-walk regular. Then
		\[|S| |T| \leq n.\]
	\end{corollary}
	\begin{proof}
		The inequality is immediate from Theorem \ref{thm:coherent} or Theorem \ref{thm:walkreg}.
	\end{proof}
	The conditions on $M$ and $N$ being positive semidefinite and satisfying either \eqref{i} or \eqref{ii} allow for a nice extension of this result.  We write $X \psd$ if $X$ is a positive semidefinite matrix. The Lovász theta graph parameter $\vartheta(G)$ is defined (see \cite{LovaszShanon}) as the optimum value of the following semidefinite program:
	\[\vartheta(G) = \max \ \{\langle J , X \rangle : X \circ A = 0 ,\ \tr X = 1,\ X \psd\}. \]
	Upon making small variations in the formulations, one obtains, respectively, the Schrijver theta \cite{SchrijverTheta} and the Szegedy theta \cite{SzegedyTheta} functions:
	\[\vartheta^-(G) = \max \ \{\langle J , X \rangle : X \circ A = 0 ,\ \tr X = 1,\ X \geq 0,\ X \psd\}. \]
	\[\vartheta^+(G) = \max \ \{\langle J , X \rangle : X \circ A \leq 0 ,\ \tr X = 1,\ X \psd\}. \]
	It is known since the first respective appearances of these thetas that, for all graphs $G$ on $n$ vertices, we have
	\begin{align}\alpha(G) \leq \vartheta^-(G) \leq \vartheta(G) \leq \vartheta^+(G) \leq \chi(\ov{G}),\end{align}
	where $\chi(G)$ is the chromatic number of $G$; and that
	\begin{align}\vartheta(G) \vartheta(\ov{G}) \geq n \quad \text{and} \quad \vartheta^-(\ov{G}) \vartheta^+(G) \geq n, \label{eq:thetas} \end{align}
	with equality in both cases of \eqref{eq:thetas} for vertex-transitive graphs. 
	
	The results in this paper extend the equality case in \eqref{eq:thetas} for graphs in homogeneous coherent configurations and 1-walk regular graphs.
	\begin{corollary} \label{cor:eqthetas}
		Let $G$ be graph on $n$ vertices that belongs to a homogeneous coherent configuration or that is 1-walk regular. Then
		\[\vartheta(G) \vartheta(\ov{G}) = n \quad \text{and} \quad \vartheta^-(\ov{G}) \vartheta^+(G) = n.\]
	\end{corollary}
	\begin{proof}
		Let $M$ and $N$ be positive semidefinite matrices that are optimal solutions for $\vartheta(G)$ and $\vartheta(\ov G)$, respectively. They satisfy condition~\eqref{i}, and therefore Theorem \ref{thm:coherent} or \ref{thm:walkreg} applies. Thus $n \geq \vartheta(G)\vartheta(\ov G)$, and then \eqref{eq:thetas} shows that equality holds. The same conclusion can be reached for matrices $N$ and $M$ which are respective optimal solutions for $\vartheta^-(\ov G)$ and $\vartheta^+(G)$, by noting that they satisfy condition \eqref{ii}.
	\end{proof}
	
	Corollary \ref{cor:eqthetas} implies that, if $M$ and $N$ are optima for $\vartheta(G)$ and $\vartheta(\ov G)$, or for $\vartheta^-(\ov G)$ and $\vartheta^+(\ov G)$, with $G$ 1-walk regular or in a homogeneous coherent configuration, then the equality characterizations from Lemmas \ref{lem1} and \ref{lem2} hold.
	
	It is interesting to consider if the tools and the framework we developed may be used to strengthen similar inequalities involving other variants of theta, such as the ones in~\cite{DukanovicRendl} and in~\cite{RobersonConicFormulations}, and also for other hierarchies of semidefinite programs (see, e.g., \cite{Laurent}).

\bibliographystyle{plain}
\bibliography{sdp-par.bib}

\begin{thebibliography}{10}

\bibitem{AnjosLasserreHandbook}
Miguel~F. Anjos and Jean~B. Lasserre, editors.
\newblock {\em {Handbook on Semidefinite, Conic and Polynomial Optimization}},
  volume 166 of {\em International Series in Operations Research {\&}
  Management Science}.
\newblock Springer US, Boston, MA, 2012.

\bibitem{ArvesonInvitation}
William Arveson.
\newblock {\em {An invitation to C*-algebras}}, volume~39.
\newblock Springer Science {\&} Business Media, 2012.

\bibitem{Bachoc}
Christine Bachoc, Dion~C. Gijswijt, Alexander Schrijver, and Frank Vallentin.
\newblock Invariant semidefinite programs.
\newblock In {\em Handbook on semidefinite, conic and polynomial optimization},
  volume 166 of {\em Internat. Ser. Oper. Res. Management Sci.}, pages
  219--269. Springer, New York, 2012.

\bibitem{BCN}
Andries~E. Brouwer, Arjeh~M. Cohen, and Arnold Neumaier.
\newblock {\em {Distance-Regular Graphs}}.
\newblock Springer-Verlag, Berlin, 1989.

\bibitem{CohnBasicALgebra}
Paul~Moritz Cohn.
\newblock {\em Basic algebra: groups, rings and fields}.
\newblock Springer Science \& Business Media, 2012.

\bibitem{DelsarteAssocSchemesCoding}
Philippe Delsarte.
\newblock {\em {An algebraic approach to the association schemes of coding
  theory}}.
\newblock PhD thesis, Universit{\'e} Catholique de Louvain, 1973.

\bibitem{DukanovicRendl}
Igor Dukanovic and Franz Rendl.
\newblock Copositive programming motivated bounds on the stability and the
  chromatic numbers.
\newblock {\em Math. Program.}, 121(2, Ser. A):249--268, 2010.

\bibitem{GodsilMeagherEKRBook}
Chris Godsil and Karen Meagher.
\newblock {\em Erd\H{o}s-{K}o-{R}ado theorems: algebraic approaches}, volume
  149 of {\em Cambridge Studies in Advanced Mathematics}.
\newblock Cambridge University Press, Cambridge, 2016.

\bibitem{Laurent}
Neboj\v{s}a Gvozdenovi\'{c} and Monique Laurent.
\newblock The operator {$\Psi$} for the chromatic number of a graph.
\newblock {\em SIAM J. Optim.}, 19(2):572--591, 2008.

\bibitem{HigmanCoherentConfig}
D.~G. Higman.
\newblock Coherent configurations. {I}. {O}rdinary representation theory.
\newblock {\em Geometriae Dedicata}, 4(1):1--32, 1975.

\bibitem{HobartBoundsSubsetsCoherent}
Sylvia~A. Hobart.
\newblock Bounds on subsets of coherent configurations.
\newblock {\em Michigan Math. J.}, 58(1):231--239, 2009.

\bibitem{HobartWillifordTightnessBoundsSubsetsCoherent}
Sylvia~A. Hobart and Jason Williford.
\newblock Tightness in subset bounds for coherent configurations.
\newblock {\em J. Algebraic Combin.}, 39(3):647--658, 2014.

\bibitem{LovaszShanon}
L\'{a}szl\'{o} Lov\'{a}sz.
\newblock On the {S}hannon capacity of a graph.
\newblock {\em IEEE Trans. Inform. Theory}, 25(1):1--7, 1979.

\bibitem{RobersonConicFormulations}
David~E. Roberson.
\newblock Conic formulations of graph homomorphisms.
\newblock {\em J. Algebraic Combin.}, 43(4):877--913, 2016.

\bibitem{SchrijverTheta}
Alexander Schrijver.
\newblock A comparison of the {D}elsarte and {L}ov\'{a}sz bounds.
\newblock {\em IEEE Trans. Inform. Theory}, 25(4):425--429, 1979.

\bibitem{SchrijverNewCodeUpperBounds}
Alexander Schrijver.
\newblock {New Code Upper Bounds From the Terwilliger Algebra and Semidefinite
  Programming}.
\newblock {\em IEEE Transactions on Information Theory}, 51(8):2859--2866, aug
  2005.

\bibitem{SzegedyTheta}
Mario Szegedy.
\newblock A note on the theta number of {L}ov{\'a}sz and the generalized
  {D}elsarte bound.
\newblock In {\em Proceedings of the 35th {A}nnual {IEEE} {S}ymposium on
  {F}oundations of {C}omputer {S}cience}, 1994.

\end{thebibliography}

\end{document}